\documentclass[12pt]{article}

\usepackage{amsmath,amssymb,amsfonts,amsthm}
\usepackage{graphics}
\usepackage{epsfig}
 \usepackage{amsmath}
\usepackage{amsfonts}
\usepackage{comment}
\allowdisplaybreaks
\font\elevensf=cmss10 scaled\magstephalf

\newtheorem{theorem} {{\elevensf THEOREM}}[section]
\newtheorem{proposition} {{\elevensf PROPOSITION}}[section]

\newtheorem{lemma} {{\elevensf LEMMA}}[section]
\newtheorem{corollary} {{\elevensf COROLLARY}}[section]
\newtheorem{example} {{\elevensf EXAMPLE}}[section]

\newtheorem{remark} {{\elevensf REMARK}}[section]

\newtheorem{definition} {{\elevensf DEFINITION}}[section]

\renewcommand\qed{$\blacksquare$}

\def\CC{{\rm \kern.24em \vrule width.02em height1.4ex depth-.05ex \kern-.26emC}}

\def\TagOnRight

\def\AA{{it I} \hskip-3pt{\tt A}}

\def\QQ{\rlap {\raise 0.4ex \hbox{$\scriptscriptstyle |$}} {\hskip -0.1em Q}}

\catcode`\@=11

\newcommand{\lb}{\left(}
\newcommand{\rb}{\right)}

\def\theequation{\@arabic{\c@section}.\@arabic{\c@equation}}

\begin{document}

\baselineskip 14pt
\parindent.4in
\catcode`\@=11 

\begin{center}

{\Huge \bf Bloch Wave Homogenization Relative to a Microstructure} \\[5mm]

{\bf Tuhin GHOSH and Muthusamy VANNINATHAN }\\[4mm]

\textit{Centre for Applicable Matematics, Tata Institute of Fundamental Research, India.}\\[2mm]

Email : \textit{vanni@math.tifrbng.res.in\ ,\ tuhin@math.tifrbng.res.in }

\end{center} 

\begin{abstract}
\noindent
In this work, we study the aspect of Bloch wave homogenization of the
new notion of convergence of microstructures represented by matrices $B^\epsilon$
related to the classical $H$-convergence of $A^\epsilon$ introduced in \cite{TG-MV-1}.
The new macro quantity $B^{\#}$ appears to incorporate the interaction between the two
microstructures $A^\epsilon,B^\epsilon$. Here we present its Bloch spectral representation 
along with the homogenization result. 
\end{abstract}


\section{Introduction}\label{c7en1}
\setcounter{equation}{0} 

We begin with recalling the new notion of convergence
related to the classical $H$-convergence introduced in \cite{TG-MV-1}.
We will address the homogenization issue with introducing the new limiting 
macro quantities and then will move into the aspect of Bloch wave homogenization
to give the spectral representation of the limiting macro quantities among the periodic microstructures.

\subsection{Convergence Relative to a Microstructure}

Let us begin by recalling the notion of $H$-convergence \cite{T}.
Let $\mathcal{M}(\alpha, \beta;\Omega)$ with $0<\alpha<\beta$ denote the set of all
real $N\times N$ symmetric matrices $A(x)$ of functions defined almost everywhere 
on a bounded open subset $\Omega$ of $\mathbb{R}^N$ such that
if $A(x)=[a_{kl}(x)]_{1\leq k,l\leq N}\in \mathcal{M}(\alpha,\beta;\Omega)$ then 
\begin{equation*} a_{kl}(x)=a_{lk}(x)\ \forall l, k=1,..,N \ \mbox{and }\   (A(x)\xi,\xi)\geq \alpha|\xi|^2,\ |A(x)\xi|\leq \beta|\xi|,\  \forall\xi \in \mathbb{R}^N,\ \mbox{ a.e. }x\in\Omega.\end{equation*}
Let $A^\epsilon$ and $A^{*}$ belong to $\mathcal{M}(a_1,a_2, \Omega)$ with 
$0 < a_1 < a_2$. We say $A^\epsilon \xrightarrow{H} A^{*}$ or $H$-converges to 
a homogenized matrix $A^{*}$, if $A^\epsilon\nabla u^\epsilon \rightharpoonup A^{*}\nabla u$
in $L^2(\Omega)$ weak, for all test sequences $u^\epsilon$ satisfying 
\begin{equation}\begin{aligned}\label{c7ad13}
u^{\epsilon} &\rightharpoonup u \quad\mbox{weakly in }H^1(\Omega)\\
-div(A^\epsilon\mathbb\nabla u^\epsilon(x))& \mbox{ is strongly convergent in } H^{-1}(\Omega).
\end{aligned}\end{equation}
Convergence of the canonical energy densities follows as a consequence:
\begin{equation}\label{c7ad12} A^\epsilon\nabla u^\epsilon\cdot\nabla u^\epsilon  \rightharpoonup  A^{*}\nabla u\cdot\nabla u \mbox{ in } \mathcal{D}^\prime(\Omega). \end{equation}
Further, their integrals over $\Omega$ converge if $u^\epsilon$ and $u$, for example, lie in $H^1_0(\Omega)$ :
\begin{equation}\label{c7ad15} \int_\Omega A^\epsilon\nabla u^\epsilon\cdot\nabla u^\epsilon dx  \rightarrow \int_\Omega A^{*}\nabla u\cdot\nabla u\ dx. \end{equation}

Here, we are concerned with other oscillating quadratic energy
densities; more precisely, let us take another sequence of matrices $B^\epsilon$
and consider the corresponding energy density :
\begin{equation*}B^\epsilon \nabla u^\epsilon\cdot \nabla u^\epsilon\end{equation*}
and study its behaviour as $\epsilon$ tends to zero. Just like \eqref{c7ad12}, 
we may expect the appearance of new macro quantities in its weak limit. This motivates
the following notion:
\begin{definition}\label{c7sid}
Let $A^\epsilon$ $H$-converge to $A^{*}$. Let $B^\epsilon$  and $B^{\#}$
be given in $\mathcal{M}(b_1,b_2;\Omega)$ where $0 < b_1 < b_2$. We say $B^\epsilon$ 
converges to $B^{\#}$ relative to $A^\epsilon$ (denoted $B^\epsilon\xrightarrow{A^\epsilon}B^{\#}$) if for all test sequences $u^\epsilon$
satisfying \eqref{c7ad13} we have 
\begin{equation}\label{c7ad17} B^\epsilon\nabla u^\epsilon\cdot\nabla u^\epsilon  \rightharpoonup  B^{\#}\nabla u\cdot\nabla u \mbox{ in } \mathcal{D}^\prime(\Omega). \end{equation}
\end{definition}
\noindent
Further, analogous to \eqref{c7ad15}, we have
\begin{equation*} \int_\Omega B^\epsilon\nabla u^\epsilon\cdot\nabla u^\epsilon dx  \rightarrow \int_\Omega B^{\#}\nabla u\cdot\nabla u dx. \end{equation*}
if $ u^\epsilon$ and $u$ belong to $H^1_0(\Omega)$.\\
\\
The significance of the limit $B^{\#}$ in the context of Calculus of Variation 
has been established in \cite{TG-MV-1}.

\noindent
Prior to that, we introduce the correctors or oscillatory test functions to define the macro quantities $A^{*},B^{*}$ and $B^{\#}$. 
\subsection{Expression of $B^{\#}$ through `oscillatory test functions' :}
\paragraph{Definition of $B^{\#}$ :}
Let $\{e_k\}_k,k=1,2,..,N$ be the standard basis vectors in $\mathbb{R}^N$.  
We define the oscillatory test functions $\chi_{k}^{\epsilon}$, $\zeta_{k}^{\epsilon}$ and 
$\psi_{k}^{\epsilon}$ in $H^1(\Omega)$ to define $A^{*}$, $B^{*}$, $B^{\#}$ 
as follows. Let $A^\epsilon$ $H$- converges to $A^{*}$, then upto a subsequence still denoted by $A^\epsilon$, there exist 
a sequence $\{\chi_k^\epsilon\}_k\in H^1(\Omega)$ such that
\begin{equation*}\begin{aligned}
\chi_{k}^{\epsilon} &\rightharpoonup 0 \mbox{ weakly in }H^1(\Omega),\\ 
 A^{\epsilon}(\nabla\chi_{k}^{\epsilon} + e_k ) &\rightharpoonup A^{*}{e_k} \mbox{ weakly in }L^2(\Omega) \mbox{ with }\\
 -div (A^{\epsilon}(\nabla\chi_{k}^{\epsilon} + e_k )) &= -div(A^{*}{e_k})\ \mbox{ in }\Omega, \ \ k=1,2,..,N.
\end{aligned}\end{equation*}
We consider the matrix $X^{\epsilon}$ defined by its columns $(\nabla\chi_{k}^{\epsilon})$
is called the corrector matrix for $A^{\epsilon}$, with the following property :
\begin{equation}\label{c7ot9}\nabla u^{\epsilon} - (X^{\epsilon}+I)\nabla u \rightarrow 0 \quad\mbox{ in }L_{loc}^1(\Omega).\end{equation}
The existence of such sequence $\{\chi_{k}^{\epsilon}\}$ is well known in homogenization theory, for more details one may look at \cite{A,T}.\\
\\
Similarly, let $B^\epsilon$ $H$-converges to $B^{*}$, then upto a subsequence still denoted by $B^\epsilon$, we define the corrector matrix $Y^{\epsilon}$ defined by its columns 
$(\nabla\zeta_{k}^{\epsilon})$ satisfying
\begin{equation*}\begin{aligned}
\zeta_{k}^{\epsilon} &\rightharpoonup 0 \mbox{ weakly in }H^1(\Omega),\\
 B^{\epsilon}(\nabla\zeta_{k}^{\epsilon} + e_k ) &\rightharpoonup B^{*}e_k \mbox{ weakly in }L^2(\Omega) \mbox{ with }\\
-div(B^{\epsilon}(\nabla\zeta_{k}^{\epsilon} + e_k )) &= -div(B^{*}e_k)\ \mbox{ in }\Omega, \ \ k=1,2,..,N.
\end{aligned}\end{equation*}
And finally we define the test functions $\psi_{k}^{\epsilon}$ bounded uniformly with respect to $\epsilon$ in $H^1(\Omega)$, satisfying
\begin{equation*} div(A^{\epsilon} \nabla \psi_{k}^{\epsilon}  - B^{\epsilon}(\nabla\chi_{k}^{\epsilon} + e_k)) =0\ \mbox{ in }\Omega, \ \ k=1,2,..,N.\end{equation*}
Such test functions $\{\psi_{k}^{\epsilon}\}$ has been introduced in \cite{KP,KV} subject to an optimal control problem.
Then upto a subsequence we consider the limit as 
\begin{equation*}\begin{aligned}
\psi_{k}^{\epsilon} &\rightharpoonup \psi_{k} \mbox{ weakly in }H^1(\Omega),\\
A^{\epsilon}\nabla \psi_{k}^{\epsilon}  - B^{\epsilon}(\nabla\chi_{k}^{\epsilon} + e_k) &\rightharpoonup\ \varsigma_k \mbox{\ (say) weakly in }L^2(\Omega)\mbox{ with }\\
div(\varsigma_k) &=0\ \mbox{ in }\Omega, \ \ k=1,2,..,N.
\end{aligned}\end{equation*}
Following that we define the limiting matrix $B^{\#}$ : For each $k=1,2,..,N$ 
\begin{equation}\label{c7Sd2}
 B^{\#}e_k :=  A^{*}\nabla \psi_{k} - \varsigma_k = A^{*}\nabla \psi_{k}- lim \{ A^{\epsilon}\nabla \psi_{k}^{\epsilon}  - B^{\epsilon}(\nabla\chi_{k}^{\epsilon} + e_k)\}; 
\end{equation}
and as a perturbation of $H$-limit $B^{*}$ we write
\begin{equation*}
B^{\#}e_k = B^{*}e_k +  A^{*}\nabla \psi_k - lim \{ A^{\epsilon} \nabla \psi_k^\epsilon - B^{\epsilon}(\nabla\chi_k^{\epsilon} - \nabla\zeta_k^{\epsilon})\}. 
\end{equation*}
\ \ \ \ (The above limits are to be understood as  $L^2(\Omega)$ weak limit).
\subsection{Convergence of flux and energy expressions in terms of $B^{\#}$ :}
Let us introduce the state equation as follows. Let $u^\epsilon\in H^1_0(\Omega)$ solves, 
\begin{equation}\label{c7Sd8}
\mbox{(State equation) : }\quad -div(A^\epsilon\nabla u^\epsilon) = f \in H^{-1}(\Omega), \mbox{ and } u^\epsilon = 0 \mbox{ on }\partial\Omega. 
\end{equation}
Let $p^\epsilon \in H^1_0(\Omega)$ solves,
\begin{equation}\label{c7sp2}
\mbox{(Adjoint-State equation) : }\quad div(A^{\epsilon}\nabla p^{\epsilon} - B^{\epsilon}\nabla u^{\epsilon}) =\ 0 \mbox{ in }\Omega, \mbox{ and } p^\epsilon = 0  \mbox{ on }\partial\Omega. 
\end{equation}
$(u^\epsilon, p^\epsilon)\rightharpoonup (u,p)$ in $H^1_0(\Omega)\times H^1_0(\Omega)$; the fluxes $\sigma^\epsilon=A^\epsilon\nabla u^\epsilon\rightharpoonup \sigma=A^{*}\nabla u$ weakly in $L^2(\Omega)$ and the new flux  $z^{\epsilon} = A^{\epsilon}\nabla p^{\epsilon} - B^{\epsilon}\nabla u^{\epsilon}
\rightharpoonup  z$ (say) weakly in  $L^2(\Omega)$.\\
\\
We give the characterization of $z$ in terms of the macro limits.\\
\\
We apply the well-known div-curl lemma \cite{T} several times to simply have the following convergences :
\begin{equation*} (A^{\epsilon}\nabla \psi_{k}^{\epsilon}- B^{\epsilon}(\nabla\chi_{k}^{\epsilon} + {e_k}))\cdot (\nabla\chi_{k}^{\epsilon} + {e_k})\rightharpoonup  (A^{*}\nabla \psi_{k}- B^{\#}e_k)\cdot e_k \ \mbox{ in }\mathcal{D}^{\prime}(\Omega)  \end{equation*}
and
\begin{equation*}A^{\epsilon}(\nabla\chi_{k}^{\epsilon} + {e_k})\cdot \nabla\psi_{k}^{\epsilon}\rightharpoonup  A^{*}e_k\cdot\nabla\psi_{k} \ \mbox{ in }\mathcal{D}^{\prime}(\Omega).  \end{equation*}
Since $A^\epsilon = (A^\epsilon)^t$ and $A^{*}=(A^{*})^t$, thus combining the above two convergences we obtain
\begin{equation}\label{c7FL12} B^{\epsilon}(\nabla\chi_{k}^{\epsilon} + {e_k})\cdot (\nabla\chi_{k}^{\epsilon} + {e_k})\rightharpoonup B^{\#}e_k\cdot e_k  \ \mbox{ in }\mathcal{D}^{\prime}(\Omega), \ k=1,2..,N.  \end{equation}
On the other hand, thanks to div-curl lemma we also have 
\begin{equation*} (A^\epsilon\nabla p^\epsilon- B^\epsilon\nabla u^\epsilon)\cdot (\nabla\chi_{k}^{\epsilon} + {e_k})\rightharpoonup  z\cdot e_k \ \mbox{ in }\mathcal{D}^{\prime}(\Omega)  \end{equation*}
and
\begin{equation*} A^\epsilon (\nabla\chi_{k}^{\epsilon} + {e_k})\cdot\nabla p^\epsilon \rightharpoonup  A^{*}e_k\cdot\nabla p \ \mbox{ in }\mathcal{D}^{\prime}(\Omega).  \end{equation*}
Thus one gets,
\begin{equation}\label{c7eir} B^\epsilon\nabla u^\epsilon\cdot (\nabla\chi_{k}^{\epsilon} + {e_k})\rightharpoonup A^{*}\nabla p\cdot e_k- z\cdot e_k \ \mbox{ in }\mathcal{D}^{\prime}(\Omega). \end{equation}
Similarly, having 
\begin{equation*} (A^{\epsilon}\nabla \psi_{k}^{\epsilon}- B^{\epsilon}(\nabla\chi_{k}^{\epsilon} + {e_k}))\cdot \nabla u^{\epsilon}\rightharpoonup  (A^{*}\nabla \psi_{k}- B^{\#}e_k)\cdot\nabla u \ \mbox{ in }\mathcal{D}^{\prime}(\Omega)  \end{equation*}
and
\begin{equation*} A^\epsilon\nabla u^\epsilon\cdot\nabla \psi_{k}^{\epsilon} \rightharpoonup  A^{*}\nabla u\cdot\nabla \psi_k \ \mbox{ in }\mathcal{D}^{\prime}(\Omega);  \end{equation*}
One obtains,
\begin{equation}\label{c7FG13} B^\epsilon (\nabla\chi_{k}^{\epsilon} + {e_k})\cdot \nabla u^\epsilon\rightharpoonup B^{\#}e_k\cdot\nabla u \ \mbox{ in }\mathcal{D}^{\prime}(\Omega). \end{equation}
Now by simply combining \eqref{c7eir} and \eqref{c7FG13}, we determine the expression of $z$ : 
\begin{equation*} z\cdot e_k = (A^{*}\nabla p -B^{\#}\nabla u)\cdot e_k, \ \ k=1,2,..N.\end{equation*}
Thus, 
$ z = A^{*}\nabla p -B^{\#}\nabla u $. Since $div\ z^\epsilon =0$ and $z^\epsilon \rightharpoonup z\ \mbox{ in }L^2(\Omega) \mbox{ weak }$, so $div\ z=0$.\\
\\
Following that, we have the energy convergence to have :
\begin{equation*} B^\epsilon\nabla u^\epsilon\cdot\nabla u^\epsilon =\ A^\epsilon\nabla u^\epsilon\cdot\nabla p^\epsilon - z^\epsilon\cdot\nabla u^\epsilon \rightharpoonup A^{*}\nabla u\cdot\nabla p - z\cdot\nabla u = B^{\#}\nabla u\cdot\nabla u\ \mbox{ in } \mathcal{D}^\prime(\Omega).\end{equation*}
Thus, we have obtained the homogenized equations as, $u,p\in H^1_0(\Omega)$ solve :
\begin{align}
&\mbox{(Homogenized State equation) : } -div(A^{*}\nabla u) = f \in H^{-1}(\Omega), \mbox{ and } u = 0 \mbox{ on }\partial\Omega ; \label{c7Sd1}\\
&\mbox{(Homogenized Adjoint-State equation) : } div(A^{*}\nabla p - B^{\#}\nabla u) =\ 0 \mbox{ in }\Omega, \mbox{ and } p = 0  \mbox{ on }\partial\Omega. \label{c7Sd7}
\end{align}

Here in this article, we would like to obtain the above limiting 
system \eqref{c7Sd1},\eqref{c7Sd7} in a different approach, while the previous methods
were based on physical space, the present one is based on Fourier techniques. 
Towards this direction, Conca and Vanninathan in \cite{CV} gave a 
new proof of weak convergence for the homogenization problem of 
elliptic operators with periodically oscillating coefficients by 
using the so called Bloch wave spectral analysis. 
Following its subsequent development in \cite{AC,COV,COV2,GV} etc.,
in a recent work \cite{TV1,TV2} we extend this idea in one non-periodic 
case of generalized Hashin-Shtrikman micro-structures. 
Here in the case of periodic microstructures we derive the homogenization results with providing
the Bloch spectral representation for the macro quantities $A^{*}$, $B^{*}$, $B^{\#}$.

The plan of the paper is as follows. In Section \ref{c7en2}, we will discuss few key properties of the new macro limit $B^{\#}$.
Following that, in Section \ref{c7en3} we will be presenting the explicit expression of $B^{\#}$ in the periodic microstructures.
In Section \ref{c7en4}, we will introduce the Bloch waves and
the Bloch spectral elements to represent the macro quantities. 
Finally, in the Section \ref{c7en5} we will establish the main homogenization results. 

In a final remark, we wish to point out here the additional difficulty in the homogenization 
of \eqref{c7sp2} via Bloch wave method. The first Bloch transform $B_1^{\epsilon}$ associated 
with $A^{\epsilon}$ does not commute with $B^{\epsilon}$. The macro quantity $B^{\#}$ is the
outcome of this non-commutativity.

\section{Properties of $B^{\#}$}\label{c7en2}
\setcounter{equation}{0} 
Let us first define $B^{\#}$ element wise i.e. to define $(B^{\#})_{lk}=B^{\#}e_l\cdot e_k.$ 
We consider the sequences  $(A^{\epsilon}\nabla\psi_{k}^{\epsilon} - B^{\epsilon}(\nabla\chi_{k}^{\epsilon} + e_k))$ and $(\nabla\chi_{l}^{\epsilon} +{e_l})$ 
where $e_k,e_l\in\mathbb{R}^N$ are the canonical basis vectors, then by applying the div-curl lemma as before,
we have the following elements wise convergence
\begin{equation}\label{c7dc2}
B^{\epsilon}(\nabla\chi_k^{\epsilon}+e_k)\cdot(\nabla\chi_l^{\epsilon}+ e_l) \rightharpoonup  B^{\#}e_k\cdot e_l\quad\mbox{in }\mathcal{D}^{\prime}(\Omega)
\end{equation}
\begin{remark}
Notice that if $A^{\epsilon}$ is independent of $\epsilon$ i.e. $A^{\epsilon} = A$ then we have $X^{\epsilon}=0$, so $B^{\#} = \overline{B}$, just a constant;
where, $\overline{B}$ is the $L^{\infty}(\Omega)$ weak* limit of $B^{\epsilon}$.
\hfill\qed
\end{remark}
\noindent
As an application of the above distributional convergence \eqref{c7dc2} one has the following :
\begin{enumerate}
\item[(i)]
If $\{B^{\epsilon}\}_{\epsilon>0}$ is symmetric, then $B^{\#}$ is also a symmetric matrix.
\item[(ii)] Let $B^{\epsilon} \in \mathcal{M}(b_1,b_2;\ \Omega)$, then the ellipticity constant of $B^{\#}$ remains same; 
and $B^{\#} \in \mathcal{M}(b_1,\widetilde{b_2};\ \Omega)$ and for some $\widetilde{b_2} \geq b_2.$ If $A^\epsilon\in\mathcal{M}(a_1,a_2;\Omega)$,
then one chooses $\widetilde{b_2} = b_2\frac{a_2}{a_1}$ (cf.\eqref{c7Sd3}).
Note that, for being $B^\epsilon$ is symmetric
the homogenized tensor $B^{*}\in\mathcal{M}(b_1,b_2;\ \Omega)$. 
\end{enumerate}
Let $\lambda=(\lambda_1,\lambda_2,..,\lambda_N)\in\mathbb{R}^N$ be an arbitrary vector, we define the corresponding oscillatory test function 
$\chi^\epsilon_\lambda= \sum_{k=1}^N \lambda_k\chi^\epsilon_k\in H^1(\Omega)$ and $\zeta^\epsilon_\lambda= \sum_{k=1}^N \lambda_k\zeta^\epsilon_k\in H^1(\Omega)$ to have 
\begin{equation}\label{c7ot11}B^{\#}\lambda\cdot\lambda =\ limit\ B^{\epsilon}(\nabla\chi_{\lambda}^{\epsilon}+ \lambda)\cdot(\nabla\chi_{\lambda}^{\epsilon}+\lambda);\end{equation}
and as a perturbation of the $H$-limit $B^{*}$,  
\begin{equation}\label{c7os18}
B^{\#}\lambda\cdot {\lambda}=\ B^{*}\lambda\cdot {\lambda} + limit\ B^{\epsilon}(\nabla\chi_{\lambda}^{\epsilon} - \nabla\zeta_{\lambda}^{\epsilon})\cdot(\nabla\chi_{\lambda}^{\epsilon} - \nabla\zeta_{\lambda}^{\epsilon}). \\                 
\end{equation}
\ \ \ (The above `limits' are to be understood in the sense of distribution.) \\
\begin{corollary}
\begin{equation}\label{c7Sd4}
B^{\#} \geq B^{*}
\end{equation}
where the equality holds if and only if $\nabla(\chi_{\lambda}^{\epsilon} - \zeta_{\lambda}^{\epsilon})\rightarrow 0 $ in $L^2(\Omega)$ for each $\lambda\in\mathbb{R}^N.$
\hfill\qed\end{corollary}
\noindent
Prior to that, in the following result, we provide the general bounds on $B^{\#}$.  
\begin{lemma}[\textbf{General Bounds}]\label{c7hsk}
Let $A^{\epsilon}\in \mathcal{M}(a_1,a_2;\Omega)$ with $0 <a_1\leq a_2 <\infty $ and $B^{\epsilon}\in\mathcal{M}(b_1,b_2;\ \Omega)$ with $0 <b_1\leq b_2 <\infty $,
$H$-converges to $A^{*}\in\mathcal{M}(a_1,a_2;\ \Omega)$ and $B^{*}\in\mathcal{M}(b_1,b_2;\ \Omega)$ respectively.
Then we have the following bounds 
\begin{equation}\label{c7Sd3} b_1I\leq \underline{B} \leq B^{*} \leq B^{\#} \leq \frac{b_2}{a_1}A^{*} \leq \frac{b_2}{a_1}\overline{A}\leq b_2\frac{a_2}{a_1}I\end{equation}
where, $(\underline{B})^{-1}$ is the $L^{\infty}$ weak* limit of the matrix sequence $(B^{\epsilon})^{-1}$ and $\overline{A}$ is the $L^{\infty}$ weak* limit of the matrix sequence $A^{\epsilon}$.
\end{lemma}
\begin{proof}
See \cite{TG-MV-1}. 
\hfill\end{proof}
\section{Integral representation of $B^{\#}$ in periodic medium}\label{c7en3}
\setcounter{equation}{0}
Here we will consider the periodic medium 
and will provide the integral representation of the macro quantities.\\
\textbf{Case (1)\ $A$ and $B$ are both periodic with same periodicity :}
Let $Y$ denotes the unit cube $[0,1]^N$ in $\mathbb{R}^N$.
Let $A(y)=[a_{kl}(y)]_{1\leq k,l\leq N}\in \mathcal{M}(a_1,a_2;\ Y)$  be such that $a_{kl}(y)$ are $Y$-periodic functions
$ \forall k,l =1,2..,N.$ Now we set 
\begin{equation*}A^{\epsilon}(x) = [a_{kl}^{\epsilon}(x)]= [a_{kl}(\frac{x}{\epsilon})]\end{equation*} and extend 
it to the whole $\mathbb{R}^N$ by $\epsilon$-periodicity with a small period of scale $\epsilon$ 
and then restricting $A^{\epsilon}$ in particular on $\Omega$ is known as periodic micro-structures.\\
The homogenized conductivity $A^{*}=[a^{*}_{kl}]$ is defined by its entries
\begin{equation}\label{c7OP8}
a^{*}_{kl} = \int_Y a_{kl}(y)(\nabla_y\chi_k + e_k)\cdot(\nabla_y\chi_l + e_l) dy 
\end{equation}
where we define the $\chi_k$ through the so-called cell-problems. 
For each unit vector $e_k$, consider the following conductivity problem in the periodic unit cell : 
\begin{equation}\label{c7hsr}
-div_y\ a_{kl}(y)(\nabla_y\chi_k(y)+e_k) = 0 \quad\mbox{in }Y,\quad y \rightarrow\chi_k(y) \quad\mbox{is $Y$-periodic. }
\end{equation}
Similarly, the homogenized conductivity $B^{*}=[b^{*}_{kl}]$ is defined by its entries
\begin{equation}\label{c7hss}
b^{*}_{kl} = \int_Y b_{kl}(y)(\nabla_y\zeta_k + e_k)\cdot (\nabla_y\zeta_l + e_l) dy 
\end{equation}
where $\zeta_k$ solves, 
\begin{equation}
-div_y\ b_{kl}(y)(\nabla_y\zeta_k(y)+e_k) = 0 \quad\mbox{in }Y,\quad y \rightarrow\zeta_k(y) \quad\mbox{is $Y$-periodic. }
\end{equation} 
Finally, we consider the cell problem for $\psi_k(y)$, satisfies  
\begin{equation}\label{c7td}
div_y\ \lb a_{kl}(y)(\nabla\psi_k ) - b_{kl}(y)(\nabla\chi_k + e_k)\rb = 0 \quad\mbox{in }Y,\quad y \rightarrow\psi_k(y) \quad\mbox{is $Y$-periodic. }
\end{equation}
We can assume that means of the cell solutions are zero, i.e. $\int_Y \chi_k dy =\int_Y \zeta_k dy = \int_Y \psi_k dy =0 $
which provides the unique solution of the cell-problems also it gives the  $L^{\infty}$ weak*  limits of the sequences $(\chi_k^{\epsilon})_{\epsilon}= (\chi_k(\frac{x}{\epsilon}))_{\epsilon}, 
(\zeta_k^{\epsilon})_{\epsilon}=(\zeta_k(\frac{x}{\epsilon}))_{\epsilon}, (\psi_k^{\epsilon})_{\epsilon}=(\psi_k(\frac{x}{\epsilon}))_{\epsilon}$ are all $0$.\\
(Since $f(\frac{x}{\epsilon})\rightharpoonup \frac{1}{|Y|}\int_Y f(y)dy$ in $L^{\infty}$ weak* for $Y$ periodic $f$.)
Thus, $B^{\#}=[b^{\#}_{kl}]$
\begin{equation*}\begin{aligned}
b^{\#}_{jk} &= lim[ b^{\epsilon}_{ij}\frac{\partial}{\partial x_i}(\chi_k^{\epsilon} + x_k) - a_{ij}^{\epsilon}\frac{\partial}{\partial x_i}(\psi_k^{\epsilon})]\\
&=\int_{Y}[b_{ij}(y)\frac{\partial}{\partial y_i}(\chi_k(y) + y_k) - a_{ij}(y)\frac{\partial}{\partial y_i}(\psi_k)]dy.
\end{aligned}\end{equation*}
By using the property \eqref{c7dc2} we can also write,
\begin{equation}\label{c7th}
b^{\#}_{jk} =  \int_{Y}b_{il}(y)\frac{\partial}{\partial y_i}(\chi_k(y) + y_k)\frac{\partial}{\partial y_l}(\chi_j(y) + y_j)dy
\end{equation}
and as a perturbation of $B^{*}$, we write
\begin{equation}\label{c7hsh} b^{\#}_{jk} =\ b^{*}_{jk} + \int_{Y}b_{il}(y)\frac{\partial}{\partial y_i}(\chi_k(y)- \zeta_k(y))\frac{\partial}{\partial y_l}(\chi_j(y) - \zeta_k(y))dy .\end{equation}
This result \eqref{c7hsh} was obtained by Kesavan and Vanninathan in \cite{KV}, there they relied on the 
asymptotic expansion method for periodic structure. Also in \cite{KP} Kesavan and Saint Jean Paulin has proved this same result by
considering oscillatory test function method.
\paragraph{Case(2)\ $A$ and $B$ are both periodic with different periodicity to one another : }
Here we would like to mention another situation where $A$ and $B$ are need not to be same $Y$ periodic function. 
Let us say $A$ is $Y_1$ periodic and $B$ is $Y_2$ periodic matrix function, i.e.
 $A(y)=[a_{kl}(y)]_{1\leq k,l\leq N}\in \mathcal{M}(a_1,a_2;\ Y_1)$  be such that $a_{kl}(y)$ are $Y_1$-periodic functions
and  $B(y)=[b_{kl}(y)]_{1\leq k,l\leq N}\in \mathcal{M}(b_1,b_2;\ Y_2)$  be such that $b_{kl}(y)$ are $Y_2$-periodic functions
$ \forall k,l =1,2..,N.$ 
Now we set 
\begin{equation*}A^{\epsilon_1}(x)= [a_{kl}(\frac{x}{\epsilon_1})]\quad\mbox{and }\ B^{\epsilon_2}(x)= [b_{kl}(\frac{x}{\epsilon_2})]\end{equation*}
and extend it to the whole $\mathbb{R}^N$ by $\epsilon_1$ and $\epsilon_2$-periodicity with a small period of scales $\epsilon_1$ and $\epsilon_2$ respectively. 
The homogenized conductivity $A^{*}$ is defined by its entries
\begin{equation*}
a^{*}_{kl} = \frac{1}{|Y_1|}\int_{Y_1}a_{kl}(y)(\nabla_y\chi_k + e_k)\cdot(\nabla_y\chi_l + e_l) dy 
\end{equation*}
where for each unit vector $(e_{k})$ the cell test function $\chi_k\in H^1(Y_1)$ solves
\begin{equation*}
-div_y\ a_{kl}(y)(\nabla_y\chi_k(y)+e_k) = 0 \quad\mbox{in }Y_{1},\quad y \rightarrow\chi_k(y) \quad\mbox{is }Y_{1}\mbox{ periodic. }
\end{equation*}
And by following the convergence result \eqref{c7dc2} we write, 
\begin{equation}\label{c7te}\begin{aligned}
b^{\#}_{jk} &= limit\ [ b^{\epsilon_2}_{il}(x)\frac{\partial}{\partial x_i}(\chi_k^{\epsilon_1}(x) + x_k)\frac{\partial}{\partial x_l}(\chi_j^{\epsilon_1}(x) + x_j)]\\
              &=limit\ [ b_{il}(\frac{x}{\epsilon_2})\frac{\partial}{\partial x_i}(\chi_k(\frac{x}{\epsilon_1}) + x_k)\frac{\partial}{\partial x_l}(\chi_j(\frac{x}{\epsilon_1}) + x_j)].
\end{aligned}\end{equation}
\ \ \ (The above `limit' is to be understood in the distributional sense).
\paragraph*{1.}
Let us say $\epsilon_1 = t(\epsilon_1,\epsilon_2).\epsilon_2$, where $t(\epsilon_1,\epsilon_2)\rightarrow t$ for some $t\in [0,\infty)$ as $\epsilon_1,\epsilon_2$ tends to $0$. 
Then we write $\frac{x}{\epsilon_2} = \frac{x}{\epsilon_1}\cdot\frac{\epsilon_1}{\epsilon_2}=\frac{x}{\epsilon_1}t + o(1)$ (through the $Y_1$, $Y_2$ periodicity) 
then from \eqref{c7te} and following \cite[Page no. 57]{A} we get
\begin{equation}\label{c7tf}
b^{\#}_{jk} =  \frac{1}{|Y_1|}\int_{Y_1}b_{il}(ty)\frac{\partial}{\partial y_i}(\chi_k(y) + y_k)\frac{\partial}{\partial y_l}(\chi_j(y) + y_j)dy.
\end{equation}
\paragraph*{2.} Similarly, if $\epsilon_2 = s(\epsilon_1,\epsilon_2).\epsilon_1$ where $s(\epsilon_1,\epsilon_2)\rightarrow s$ for some $s \in [0,\infty)$ as $\epsilon_1,\epsilon_2$ tends to $0$, 
then we write $\frac{x}{\epsilon_1} = \frac{x}{\epsilon_2}\cdot\frac{\epsilon_2}{\epsilon_1}=\frac{x}{\epsilon_2}s + o(1)$ and we get
\begin{equation}\label{c7tg}
b^{\#}_{jk} =  \frac{1}{|Y_2|}\int_{Y_2}b_{il}(y)\frac{\partial}{\partial y_i}(\chi_k(sy) + y_k)\frac{\partial}{\partial y_l}(\chi_j(sy) + y_j)dy.
\end{equation}

\subsection{Variational characterizations of $B^{\#}$}
The homogenized matrix $A^{*}$ is defined in terms of the solutions of the
cell problems \eqref{c7hsr}. When $A(y)$ is symmetric, it is convenient to give 
another definition of $A^{*}$ involving standard variational principles. 
We consider the quadratic form $A^{*}\lambda\cdot\lambda$ where $\lambda$
is any constant vector in $\mathbb{R}^N$
\begin{equation*}
 A^{*}\lambda\cdot\lambda= \int_Y A(y)(\nabla_y\chi_\lambda + \lambda)\cdot(\nabla_y\chi_\lambda + \lambda) dy, 
\end{equation*}
where $\chi_\lambda$ is the solution of the following cell problem :
\begin{equation}\label{c7hsx}
-div_yA(y)(\nabla_y\chi_\lambda(y)+\lambda) = 0 \quad\mbox{in }Y,\quad y \rightarrow\chi_\lambda(y) \quad\mbox{is $Y$-periodic. }
\end{equation}
It is well-known that equation \eqref{c7hsx} is the Euler-Lagrange equation of the
following variational principle: Find $w(y)$ that minimizes
\begin{equation*} \int_Y A(y)(\nabla_y w(y) + \lambda)\cdot(\nabla_y w(y) + \lambda) dy\end{equation*} 
over all periodic functions $w$. Hence $A^{*}\lambda\cdot\lambda$ is given by the minimization 
of the potential energy
\begin{equation}\label{c7hsi}
 A^{*}\lambda\cdot\lambda=\ \underset{w(y)\in H^1_{\#}(Y)}{min}\int_Y A(y)(\nabla_y w(y) + \lambda)\cdot(\nabla_y w(y) + \lambda) dy. 
\end{equation}
Similarly, for the macro quantity $B^{*}$ it is given by 
\begin{equation}\label{c7hsj}
 B^{*}\lambda\cdot\lambda=\ \underset{w(y)\in H^1_{\#}(Y)}{min}\int_Y B(y)(\nabla_y w(y) + \lambda)\cdot(\nabla_y w(y) + \lambda) dy. 
\end{equation}
Next we give the variational characterization for the macro quantity $B^{\#}$. For any $\lambda\in \mathbb{R}^N$ 
we find $\psi_\lambda$ from below, while $\chi_\lambda$ is satisfying \eqref{c7hsx}
\begin{equation}\label{c7hsy}
-div_y(A(y)\nabla\psi_\lambda  - B(y)(\nabla\chi_\lambda + \lambda)) = 0 \quad\mbox{in }Y,\quad y \rightarrow\psi_\lambda(y) \quad\mbox{is $Y$-periodic. }
\end{equation}
Then following \eqref{c7th}, the quadratic form $B^{\#}\lambda\cdot\lambda$, associated with the macro quantity $B^{\#}$, is defined by
\begin{equation}\label{c7ti}
B^{\#}\lambda\cdot\lambda=\ \int_Y B(y)(\nabla_y \chi_\lambda(y) + \lambda)\cdot(\nabla_y \chi_\lambda(y) + \lambda) dy.  
\end{equation}
We consider two bilinear forms associated with the matrices $A(y)$ and $B(y)$ :
\begin{align*}
 &a : H^1_{\#}(Y)\times H^1_{\#}(Y) \rightarrow \mathbb{R} \ \mbox{ defined as } a(w_1,w_2)= \int_Y A(y)\nabla_y w_1(y)\cdot \nabla_y w_2(y) dy  \\
\mbox{ and }\ &b : H^1_{\#}(Y)\times H^1_{\#}(Y) \rightarrow \mathbb{R} \ \mbox{ defined as } b(w_1,w_2)= \int_Y B(y)\nabla_y w_1(y)\cdot \nabla_y w_2(y) dy.
\end{align*}
Then there exists a constant $a_1 >0$ (ellipticity constant of the matrix $A(y)$) such that 
\begin{equation}\label{c7dc4}
 \underset{w_1\in H^1_{\#}(Y)}{inf}\lb\underset{w_2\in H^1_{\#}(Y)}{sup}\ \frac{a(w_1,w_2)}{||\nabla w_1||_{L^2_{\#}}||\nabla w_2||_{L^2_{\#}}}\rb \geq\ a_1\ ; 
\end{equation}
 which is known as so called Babu\v{s}ka-Brezzi condition \cite[(i), Lemma 4.1]{GR}.  
And the bilinear form $b(.,.)$ is $H^1_{\#}(Y)/\mathbb{R}$-elliptic, i.e. there exist a $b_1 >0 $ (ellipticity constant of the matrix $B(y)$) such that 
\begin{equation}\label{c7dc5}
b(w,w) \geq \ b_1||w||^2_{H^1_{\#}(Y)/\mathbb{R}},\ \ \forall w\in H^1_{\#}(Y)/\mathbb{R}.
\end{equation}
Next we define the Lagrangian 
\begin{equation}
\mathcal{L}_\lambda :  H^1_{\#}(Y)\times H^1_{\#}(Y) \rightarrow \mathbb{R} \ \mbox{ defined as } \mathcal{L}_\lambda(w_1,w_2)= b(w_1+\lambda\cdot y,w_1+\lambda\cdot y) + a(w_1+\lambda\cdot y,w_2).  
\end{equation}
Then following \cite[Theorem 4.2]{GR} together with the conditions \eqref{c7dc4} and \eqref{c7dc5}, the solution  
$\chi_\lambda,\psi_\lambda$ of \eqref{c7hsx} and \eqref{c7hsy} is characterized by :
\begin{equation}
 \underset{w_1\in H^1_{\#}(Y)}{Min}\lb \underset{w_2\in H^1_{\#}(Y)}{Max} \mathcal{L}_\lambda(w_1,w_2)\rb = \mathcal{L}_\lambda(\chi_\lambda,\psi_\lambda) = \underset{w_2\in H^1_{\#}(Y)}{Max}\lb \underset{w_1\in H^1_{\#}(Y)}{Min} \mathcal{L}_\lambda(w_1,w_2)\rb. 
\end{equation}
Now as we see, by multiplying \eqref{c7hsx} by $\psi_\lambda$,
\begin{align*}
\mathcal{L}_\lambda(\chi_\lambda,\psi_\lambda) &=\ \int_Y B(y)(\nabla_y\chi_\lambda + \lambda)\cdot (\nabla_y\chi_\lambda + \lambda)dy + \int_Y A(y)(\nabla_y\chi_\lambda +\lambda)\cdot\nabla_y\psi_\lambda dy \\
                                       &=\ \int_Y B(y)(\nabla_y\chi_\lambda + \lambda)\cdot (\nabla_y\chi_\lambda + \lambda)dy. 
\end{align*}                                       
So by \eqref{c7ti} we obtain
\begin{equation*}\mathcal{L}_\lambda(\chi_\lambda,\psi_\lambda)=\ B^{\#}\lambda\cdot\lambda.\end{equation*}
Thus the variational characterization of $B^{\#}\lambda\cdot\lambda$ would be
\begin{equation}\begin{aligned}
&B^{\#}\lambda\cdot\lambda\\
&= \underset{w_1\in H^1_{\#}(Y)}{Min}\lb \underset{w_2\in H^1_{\#}(Y)}{Max}\int_Y B(y)(\nabla_y w_1 + \lambda)\cdot (\nabla_y w_1 + \lambda)dy + \int_Y A(y)(\nabla_y w_1 +\lambda)\cdot\nabla_y w_2 dy \rb
\end{aligned}\end{equation}
or, 
\begin{equation}\begin{aligned}
&B^{\#}\lambda\cdot\lambda\\
&= \underset{w_2\in H^1_{\#}(Y)}{Max}\lb \underset{w_1\in H^1_{\#}(Y)}{Min}\int_Y B(y)(\nabla_y w_1 + \lambda)\cdot (\nabla_y w_1 + \lambda)dy + \int_Y A(y)(\nabla_y w_1 +\lambda)\cdot\nabla_y w_2 dy\rb.
\end{aligned}\end{equation}
\section{Bloch wave spectral analysis}\label{c7en4}
\setcounter{equation}{0}
Here we will present a brief survey on Bloch waves and corresponding 
Bloch eigen elements in periodic structures. 
Based on that, we will give the spectral representation of the macro coefficients and 
will make the passage into passing to the limit. 
\subsection{Introduction to Bloch waves, Bloch eigenvalues and eigenvectors in periodic structures}\label{c7Sd9}
Introduction of Bloch waves in periodic structures is based on Floquet
principle: periodic structures can be regarded as multiplicative
perturbations of homogeneous media by periodic functions. This principle
give rise to a new class of functions, namely $(\eta, Y_1)$ periodic functions.
\begin{align*}
\psi(.;\eta) \mbox{ is }(\eta;Y_1)-\mbox{periodic, if }
\psi(y+2\pi m;\eta) = e^{2\pi i m.\eta}\psi(y;\eta) \quad \forall m \in \mathbb{Z}^N , y \in \mathbb{R}^N.
\end{align*}
Accordingly, the following spaces are used: $L^2_{\#}(\eta,Y_1),
L^2_{\#}(Y_1),H^1_{\#}(\eta,Y_1),H^1_{\#}(Y_1)$ etc. 
\begin{align*}
  L^2_{\#}(\eta;Y_1) &= \{ v\in L^2_{loc}(\mathbb{R}^N)\ |\   v \mbox{ is }(\eta,Y_1)\ - \mbox{ periodic}\},\\
  H^1_{\#}(\eta; Y_1) &= \{ v\in H^1_{loc}(\mathbb{R}^N)\ |\ v \mbox{ is }(\eta,Y_1)\ - \mbox{ periodic}\}.
\end{align*}
The $(\eta,Y_1)$ periodicity condition remains unaltered if we
replace $\eta$ by $(\eta + q)$ with $q \in\mathbb{Z}^N$ , so $\eta$ 
can therefore be confined to the dual cell $\eta \in Y_1^{\prime} = (2\pi)[-\frac{1}{2},\frac{1}{2}]^N$
or equivalently dual torus. $\eta=0$ gives back the usual periodicity condition.\\
In fact, these classes are state spaces for Bloch waves.
The link between `Bloch waves' and the traditional `Homogenization theory' is as
follows: the homogenized tensor and the cell test functions can be
obtained as infinitesimal approximation from (ground state) Bloch waves at $\eta =0$
and its energy.\\
\\
We consider the operator 
\begin{equation*} \mathcal{A} \equiv -\frac{\partial}{\partial y_k}\lb a_{kl}(y)\frac{\partial}{\partial y_l} \rb \quad k,l=1,2..,N\end{equation*}
where the coefficient matrix  $A(y) = [a_{kl}(y)]$ defined on $Y_1$ a.e. where $Y_1 =[0,1]^N$ is known as the periodic cell
and $A\in\mathcal{M}(a_1,a_2;\ Y_1)$ for some $0< a_1 \leq a_2 < \infty$, i.e. 
\begin{equation*}a_{kl}=a_{lk}\hspace{5pt}\forall k,l \mbox{ and }\ (A(y)\xi,\xi) \geq \alpha|\xi|^2,\ \ |A(y)\xi|\leq \beta|\xi|\mbox{ for any } \xi \in \mathbb{R}^N,\mbox{ a.e. on }Y_1.\end{equation*}
We define Bloch waves $\psi$ associated with the operator 
$\mathcal{A}$ as follows. Let us consider the following spectral problem 
parameterized by $\eta \in \mathbb{R}^N$ : \\
\\
Find $\lambda = \lambda(\eta) \in \mathbb{R}$ 
and $\psi_A = \psi_A(y; \eta)$ (not identically zero) such that
\begin{equation}\label{c7b1}
\mathcal{A}\psi_A(.;\eta) =\lambda(\eta)\psi_A(.;\eta) \quad\mbox{in }\mathbb{R}^N,\ \ \ \psi_A(.;\eta) \mbox{ is }(\eta;Y_1)-\mbox{periodic.}
\end{equation}
By applying the Floquet principle we define $\varphi_A(y; \eta) = e^{-iy·\eta} \psi_A(y; \eta)$,  and \eqref{c7b1} can be rewritten in terms of 
$\varphi_A$ as follows:
\begin{equation}\label{c7ps}
\mathcal{A}(\eta)\varphi_A = \lambda(\eta) \varphi_A \mbox{ in }\mathbb{R}^N, \quad \varphi_A \mbox{ is $Y_1$-periodic}
\end{equation}
where, the operator $\mathcal{A}(\eta)$  is defined by  
\begin{equation*}
\mathcal{A}(\eta) = -(\frac{\partial}{\partial y_k} + i\eta_k)[a_{kl}(y)(\frac{\partial}{\partial y_l} + i\eta_l)].
\end{equation*}
It is well known from \cite{CPV} that for each $\eta\in Y_1^{\prime}$ ,
the above spectral problem admits a discrete sequence of eigenvalues with the following properties:
\begin{equation*}
0 \leq \lambda_1 (\eta) \leq · · · \leq \lambda_m (\eta) \leq · · · \rightarrow \infty, \mbox{ and } \forall\ m \geq 1,\ \lambda_m(\eta)\mbox{ is a Lipschitz function of }\eta \in Y_1^{\prime} .
\end{equation*}
The corresponding eigenfunctions denoted by $\psi_{m,A}(.; \eta)$ and $\varphi_{m,A}(.; \eta)$ 
form orthonormal bases in the spaces of all $L^2_{loc} (\mathbb{R}^N )$
functions which are $(\eta; Y_1 )$-periodic and $Y_1$ -periodic, respectively; 
In fact these eigenfunctions belong to the spaces $H^1_{\#}(\eta; Y_1)$ and $H^1_{\#}(Y_1)$ respectively.\\
\\
Similarly, we consider the translated operator $\mathcal{B}(\eta)= -(\frac{\partial}{\partial y_k}+i\eta_k)(b_{kl}(y)(\frac{\partial}{\partial y_l}+i\eta_l)) $ 
associated with $B=[b_{kl}(y)]$ a symmetric $Y_2$ periodic matrix in $\mathcal{M}(b_1,b_2;\ Y_2)$
and we define the Bloch eigenvalues $\mu_m(\eta)$ and Bloch eigenvectors $\varphi_m(.;\eta)$ as
\begin{equation}\label{c7Sd5}
\mathcal{B}(\eta)\varphi_{m,B}(.;\eta) =\mu_m(\eta)\varphi_{m,B}(.;\eta) \mbox{ in }\mathbb{R}^N,\quad \varphi_{m,B}(.;\eta) \mbox{ is }Y_2-\mbox{periodic.}
\end{equation}
\\
To obtain the spectral resolution of $\mathcal{A}^{\epsilon_1}=-\frac{\partial}{\partial x_k}(a_{kl}(\frac{x}{\epsilon_1})\frac{\partial}{\partial x_l})$ in an analogous manner, 
let us introduce Bloch waves at the $\epsilon_1$-scale:
\begin{equation*}\lambda^{\epsilon_1}_m(\xi) = \epsilon_1^{-2}\lambda_m(\eta),\quad \varphi^{\epsilon_1}_{m,A}(x;\xi) = \varphi_{m,A} (y;\eta),
\ \ \psi^{\epsilon_1}_{m,A}(x;\xi) = \psi_{m,A} (y; \eta)\end{equation*}
where the variables $(x, \xi)$ and $(y;\eta)$ are related by $y = \frac{x}{\epsilon_1}$ and $\eta =\epsilon_1\xi$.
Observe that $\varphi^{\epsilon_1}_{m,A}(x;\xi)$ is $\epsilon_1 Y_1$ -periodic (in $x$) and $\epsilon_1^{-1}Y_1^{\prime}$ -periodic 
with respect to $\xi$. In the same manner, $\psi^{\epsilon_1}_m(.; \xi)$ is $(\epsilon_1\xi; \epsilon Y_1 )$-periodic. 
The dual cell at $\epsilon$-scale is $\epsilon^{-1}Y_1^{\prime}$ where $\xi$ varies.\\
The functions $\psi_{m,A}^{\epsilon_1}$ and $\varphi_{m,A}^{\epsilon_1}$ 
(referred to as Bloch waves) enable us to describe the spectral resolution of 
$\mathcal{A}^{\epsilon_1}$ ( an unbounded self-adjoint operator
in $L^2 (\mathbb{R}^N)$ ) in the orthogonal basis 
$\{e^{ix\cdot\xi} \varphi_{m,A}^{\epsilon}(x; \xi)\ |\ m \geq 1, \xi \in \epsilon_1^{-1}Y_1^{\prime} \}$. 
More precisely, we have the following.
\begin{proposition}[Bloch decomposition \cite{CV}]\label{c7deA}
Let $g \in L^2(\mathbb{R}^N)$. The $m$-th Bloch coefficient of $g$ at the $\epsilon_1$-scale is
defined as follows:
\begin{equation}\label{c7bt}
(B^{\epsilon_1}_{m,A}\ g)(\xi) = \int_{\mathbb{R}^N} g(x)e^{-ix\cdot\xi} \overline{\varphi^{\epsilon_1}}_{m,A}(x;\xi)dx \quad \forall m \geq 1,\ \xi \in \epsilon_1^{-1}Y_1^{\prime}.
\end{equation}
Then the following inverse formula holds:
\begin{equation}\label{c7ibt}
g(x) = \int_{\epsilon_1^{-1}Y_1^{\prime}}(B^{\epsilon_1}_{m,A}\ g)(\xi)e^{ix\cdot\xi} \varphi^{\epsilon_1}_{m,A}(x; \xi)d\xi.
\end{equation}
And the Parseval’s identity:
\begin{equation*}\int_{\mathbb{R}^N}|g(x)|^2 dx = \int_{\epsilon_1^{-1}Y_1^{\prime}}\sum_{m=1}^{\infty}|(B^{\epsilon_1}_{m,A}\ g)(\xi)|^{2} d\xi.\end{equation*}
Finally, for all $g$ in the domain of $\mathcal{A}^{\epsilon_1}$, we have
\begin{equation*}\mathcal{A}^{\epsilon_1}g(x) = \int_{\epsilon_1^{-1}Y_1^{\prime}}\sum_{m=1}^{\infty}\lambda^{\epsilon_1}_m(\xi)(B^{\epsilon_1}_{m,A}\ g)(\xi)e^{ix\cdot\xi} \varphi^{\epsilon_1}_{m,A}(x;\xi)d\xi.\end{equation*}
\end{proposition}

Using the above proposition, the classical homogenization result was deduced in \cite{CV}. 
We recall the main steps. The first one consists of considering a sequence
$u^{\epsilon_1} \in H^1(\mathbb{R}^N)$ satisfying  $\mathcal{A}^{\epsilon_1}u^{\epsilon} = f $ in $\mathbb{R}^N$ 
with the fact $u^{\epsilon_1}\rightharpoonup u$ in $H^{1}(\mathbb{R}^N)$ 
weak and $u^{\epsilon_1}\rightarrow u $ in $L^{2}(\mathbb{R}^N)$ strong.
We can express the equation in the equivalent form
\begin{equation*} \lambda^{\epsilon_1}_m(\xi)(B^{\epsilon_1}_{m,A}\ u^{\epsilon_1})(\xi) = (B^{\epsilon_1}_{m,A}\ f )(\xi) \quad \forall m \geq 1, \xi \in \epsilon^{-1}Y_1^{\prime}\end{equation*}
In the homogenization process, one can neglect all the relations for $m \geq 2$. More
precisely, it is proved in \cite{CV} that the following result holds.
\begin{proposition}\cite{CV}\label{c7hm2}
\begin{equation*}
|| \int_{\epsilon_1^{-1}Y_1^{\prime}}\sum_{m=2}^{\infty} (B^{\epsilon_1}_{m,A}\ u^{\epsilon_1}(\xi))\ e^{ix\cdot\xi} \varphi^{\epsilon_1}_{m,A}(x;\xi)d\xi ||_{L^2(\mathbb{R}^N)}\leq c\epsilon_1. 
\end{equation*}
\hfill\qed
\end{proposition}
\noindent
Thus we can concentrate our attention only on the relation corresponding to the first Bloch wave:
\begin{equation}\label{c7fe} \lambda_1^{\epsilon_1}(\xi)(B_{1,A}^{\epsilon_1}\ u^{\epsilon_1})(\xi) =\ B^{\epsilon_1}_{1,A}\ f(\xi) \quad\forall \xi\in \epsilon_1^{-1}Y_1^{\prime}.\end{equation}
\noindent
Similarly, by considering the sequence
$p^{\epsilon_1} \in H^1(\mathbb{R}^N)$ satisfying the adjoint state 
equation $\mathcal{A}^{\epsilon_1}p^{\epsilon_1} = \mathcal{B}^{\epsilon_2}u^{\epsilon_1}$ in $\mathbb{R}^N$
with the fact $p^{\epsilon_1}\rightharpoonup p$ in $H^{1}(\mathbb{R}^N)$ weak and  
$p^{\epsilon_1}\rightarrow p$ in $L^{2}(\mathbb{R}^N)$strong.
Then applying the first Bloch transformation \eqref{c7bt} on the adjoint state equation we have,
\begin{equation}\label{c7afe}
\lambda_1^{\epsilon_1}(\xi)(B_{1,A}^{\epsilon_1}\ p^{\epsilon_1})(\xi) =\ (B^{\epsilon_1}_{1,A}\ ( \mathcal{B}^{\epsilon_2}\ u^{\epsilon_1} ))(\xi) \quad\forall \xi\in \epsilon_1^{-1}Y_1^{\prime}.
\end{equation}
By passing to the limit as $\epsilon_1\rightarrow 0$ in \eqref{c7fe}, 
we get the homogenized equation in the Fourier space
\begin{equation}\label{c7forA}
a^{*}_{kl}\ \xi_k\xi_l\ \widehat{u}(\xi) = \widehat{f}(\xi) \quad\forall\xi \in \mathbb{R}^N.
\end{equation}
On the other hand, passing to the limit as $\epsilon_1,\epsilon_2 \rightarrow 0$ in \eqref{c7afe}, we would like to get the adjoint state homogenized equation in the Fourier space
\begin{equation*}
a^{*}_{kl}\ \xi_k\xi_l\ \widehat{p}(\xi)= b^{\#}_{kl}\ \xi_k\xi_l\ \widehat{u}(\xi) \quad\forall\xi \in \mathbb{R}^N
\end{equation*}
which we present as a new result in the next section (cf. \eqref{c7pst}).\\ 
\\
However, the result for \eqref{c7forA} follows from the following regularity result of the Bloch eigen value  $\lambda_1(\eta)$ 
and the Bloch eigen vector $\varphi_{1,A}(\eta)$ of the operator $\mathcal{A}(\eta)$.
\begin{proposition}[Regularity of the Ground state \cite{CV,COV}]\label{c7based}
Under the periodic assumption on the matrix $A\in \mathcal{M}(a_1,a_2; Y_1)$, then there exists a $\delta > 0$
such that the first eigenvalue $\lambda_1(\eta)$ is an analytic function on 
$B_{\delta}(0) = \{\eta\in\mathbb{R}^N \ |\ |\eta| < \delta\}$, and
there is a choice of the first eigenvector $\varphi_1(y;\eta)$ satisfying
\begin{equation*}
\eta \mapsto \varphi_{1,A}(.;\eta) \in H^1_{\#}(Y_1)\mbox{ is analytic on } B_{\delta}\quad\mbox{ and }\ \varphi_{1,A}(y;0)= |Y_1|^{-1/2}
\end{equation*}
with the usual normalization condition and the choice of the phase factor to determine the eigenvector uniquely,
\begin{equation*}
||\varphi_{1,A}(.;\eta)||_{L^2(Y_1)}=1, \mbox{ and } \quad \Im \int_{Y_1} \varphi_{1,A}(y;\eta) dy = 0,\quad \eta\in B_{\delta}.
\end{equation*}
Moreover, we have the following relations,
\begin{equation*}\begin{aligned}
& \lambda_1(0) = 0,\quad D_k \lambda_1(0) = \frac{\partial\lambda_1}{\partial \eta_k}(0) = 0 \quad\forall k = 1,2,..,N.\\
& \varphi_{1,A}(.;0)= |Y_1|^{-1/2},\quad D_k\varphi_{1,A}(.;0) = i|Y_1|^{-1/2}\chi_k(y)\\ 
&\frac{1}{2}D^2_{kl}\lambda_1(0) = \frac{1}{2}\frac{\partial^2\lambda_1}{\partial \eta_k\partial\eta_l}(0) = a^{*}_{kl}. \quad\forall k,l =1,2,..,N
\end{aligned}\end{equation*}
- the last expression is considered as a Bloch spectral representation of the homogenized matrix $A^{*}$. 
\hfill\qed\\
\end{proposition}

Following that, here we give the desired Bloch spectral representation of the other two macro quantities $B^{*}$ and $B^{\#}$. 
\paragraph{Bloch Spectral representation of $B^{*}$ :}
Similarly, the eigen elements $\mu_1(\eta)$ and $\varphi_1(y;\eta)$ of the operator
$\mathcal{B}(\eta)$ defined in \eqref{c7Sd5} provide  Bloch spectral representation of the homogenized matrix $B^{*}$. 
\begin{equation}\label{c7Sd6}\begin{aligned}
&\mu_1(0) = D_k \mu_1(0) = 0, \quad  \varphi_{1,B}(.;0)= |Y_2|^{-1/2},\ \ D_k\varphi_{1,B}(.;0) = i|Y_2|^{-1/2}\zeta_k(y)\\
&\frac{1}{2}D^2_{kl}\mu_1(0)=\ b^{*}_{kl} \quad\forall k,l =1,2,..,N.
\end{aligned}\end{equation}
\paragraph{Bloch Spectral representation of $B^{\#}$ :}
\paragraph{Case $(1)\ A$ and $B$ are periodic with the same periodicity $Y$ :}
Let us define the following map
\begin{equation}\begin{aligned}\label{c7rwt}
&\nu_1 : B(\delta) \mapsto \mathbb{C} \mbox{ by }\ \nu_1(\eta) = \langle \mathcal{B}(\eta)\varphi_{1,A}(.;\eta), \varphi_{1,A}(.;\eta)\rangle_{L^2(Y)} \\
\mbox{i.e. }\ &\nu_1(\eta)=\ \frac{1}{|Y|}\int_{Y} B(y)(\nabla_y\varphi_{1,A}(y;\eta) + i\eta\varphi_{1,A}(y;\eta) )\cdot(\nabla_y\varphi_{1,A}(y;\eta)+ i \eta\varphi_{1,A}(y;\eta)) dy
\end{aligned}\end{equation}
where $\varphi_{1,A}(y;\eta)$ is the first Bloch eigen vector of the operator $\mathcal{A}(\eta)$ defined in \eqref{c7ps} and 
$\mathcal{B}(\eta)=-(\frac{\partial}{\partial y_k}+i\eta_k)(b_{kl}(y)(\frac{\partial}{\partial y_l}+i\eta_l)) $ 
is the translated operator associated with $B=[b_{kl}(y)]$.\\
\\
Now clearly $\eta \mapsto \nu_1(\eta)$ is an analytic function on $B_{\delta}$ 
as $\eta \mapsto \varphi_{1,A}(.;\eta) \in H^1_{\#}(Y)$ is analytic on $B_{\delta}$.
We compute upto second order derivatives of $\nu_1(\eta)$ at the origin based on the proposition \eqref{c7based} 
and identify the matrix $B^{\#}=[b^{\#}_{kl}]$ with $\frac{1}{2}D^2_{kl}\nu_1(0)$ for all $k,l =1,2,..,N$.\\
\\
Let us rewrite \eqref{c7rwt} by using $||\varphi_{1,A}(.;\eta)||_{L^2(Y)} = 1, \quad\eta\in B(\delta)$ 
\begin{equation}\label{c7nuq}
\langle (\mathcal{B}(\eta)-\nu(\eta))\varphi_{1,A}(.;\eta), \varphi_{1,A}(.;\eta)\rangle_{L^2(Y)} = 0            
\end{equation}
together with 
\begin{align*}
\mathcal{B}(\eta) &=\ \mathcal{B} + i\eta_k E_k + \eta_k\eta_l b_{kl}(y)\\
\mbox{where, }\quad \mathcal{B}=-\frac{\partial}{\partial y_k}(b_{kl}(y)\frac{\partial}{\partial y_l})& \quad\mbox{and }\ E_k(\varphi) = - b_{kj}(y)\frac{\partial \varphi}{\partial y_j} - \frac{\partial}{\partial y_j}(b_{kj}(y)\varphi).
\end{align*}
\paragraph{Step(i)  Zeroth order derivative of $\nu_1$ at $0$ :} As $\varphi_{1,A}(y;0) = |Y|^{-1/2}$ it implies $\nu_1(0) = 0.$
\paragraph{Step(ii) First order derivatives of $\nu_1$ at $0$ :}
By differentiating \eqref{c7nuq} once with respect to $\eta_k$ at origin we obtain
\begin{equation*}\begin{aligned}
&\langle D_k\{(\mathcal{B}(0)- \nu_1(0))\}\varphi_{1,A}(.;0),\varphi_{1,A}(.;0)\rangle + \langle (\mathcal{B}(0) - \nu_1(0))D_k\varphi_{1,A}(.;0),\varphi_{1,A}(.;0)\rangle \\
 &+ \langle (\mathcal{B}(0))-\nu_1(0)\varphi_{1,A}(.;0), D_k\varphi_{1,A}(.;0)\rangle =\ 0. 
\end{aligned}\end{equation*}
Now by using $D_k \mathcal{B}(0) = iE_k $, and $\varphi_{1,A}(.;0) = |Y|^{-1/2}$ is independent of $y$, we get 
\begin{equation*}
D_k \nu_1(0) =\ 0,\quad \forall k = 1,2..,N.
\end{equation*}
\paragraph{Step(iii) Second derivatives of $\nu_1$ at $0$ :}  We differentiate \eqref{c7nuq} twice with respect to $\eta_k$ and $\eta_l$ at origin to obtain
\begin{equation*}\begin{aligned}
&\langle [D^2_{kl} (\mathcal{B}(0) - \nu_1 (0))]\varphi_{1,A} (.;0),\varphi_{1,A}(.;0)\rangle + \langle [D_k(\mathcal{B}(0) - \nu_1(0))]D_l\varphi_{1,A}(.;0),\varphi_{1,A}(.;0)\rangle\\
&+ \langle (\mathcal{B}(0) - \nu_1(0))D^2_{kl}\varphi_{1,A} (.;0),\varphi_{1,A}(.;0)\rangle + \langle [D_k(\mathcal{B}(0) - \nu_1(0))]\varphi_{1,A}(.;0),D_l\varphi_{1,A}(.;0)\rangle \\
&+ \langle (\mathcal{B}(0) - \nu_1(0))D_k\varphi_1(.;0),D_l\varphi_{1,A}(.;0)\rangle + \langle (\mathcal{B}(0) - \nu_1(0))\varphi_{1,A} (.;0),D^2_{kl}\varphi_{1,A}(.;0)\rangle =\ 0.
\end{aligned}\end{equation*}
Now by using $\nu_1(0)= D_k \nu_1(0) = 0$ and $\varphi_{1,A}(.;0)= |Y_1|^{-1/2}, D_k \varphi_{1,A}(.;0) = i|Y_1|^{-1/2}\chi_k(y)$
 and $D_k \mathcal{B}(0) = iE_k,  D^2_{kl} = b_{kl}(y)$  we get, 
\begin{equation}\label{c7rpb1}
\frac{1}{2}D^2_{kl}\nu_1(0) = \frac{1}{|Y|}\int_Y b_{kl}(y)dy -\frac{1}{2|Y|}\int_Y (E_k\chi_l(y)+ E_l\chi_k(y))dy  = b^{\#}_{kl}
\end{equation}
-where  $b^{\#}_{kl}$ are precisely the macro coefficients defined in \eqref{c7th}. 
The above expression \eqref{c7rpb1} is to be considered as the Bloch spectral representation of the limit matrix $B^{\#}$.
\paragraph{Case $2.\quad$ $A$ and $B$ are different $Y_1$ and $Y_2$ periodic respectively :} 
We assume that $a^{\epsilon_1}_{kl} = a_{kl}(\frac{x}{\epsilon_1})$ and $b^{\epsilon_2}_{kl} = b_{kl}(\frac{x}{\epsilon_2})$
with the fact  $\frac{\epsilon_1}{\epsilon_2} \rightarrow t \in [0,\infty)$ as $\epsilon_1,\epsilon_2 \rightarrow 0$, then we define
\begin{equation*}
\nu^t_1(\eta) =\frac{1}{|Y_1|}\int_{Y} B(ty)(\nabla_y\varphi_{1,A}(y;\eta) + i\eta \varphi_{1,A}(y;\eta))\cdot(\nabla_y\varphi_{1,A}(y;\eta)+ i \eta \varphi_{1,A}(y;\eta)) dy.
\end{equation*}
If $\frac{\epsilon_2}{\epsilon_1} \rightarrow s \in [0,\infty)$ as $\epsilon_1,\epsilon_2 \rightarrow 0$, then we define 
\begin{equation*}
\nu^s_1(\eta) =\frac{1}{|Y_2|}\int_{Y_2} B(y)(\nabla_y\varphi_{1,A}(sy;\eta) + i\eta\varphi_{1,A}(sy;\eta) )\cdot(\nabla_y\varphi_{1,A}(sy;\eta)+ i \eta\varphi_{1,A}(sy;\eta)) dy.
\end{equation*}
Following the same above calculations we get $\frac{1}{2}D^2_{kl}\nu^t_1(\eta)|_{\eta =0}= b^{\#}_{kl}$ which is given by 
\eqref{c7tf} and $\frac{1}{2}D^2_{kl}\nu^s_1(\eta)|_{\eta =0}= b^{\#}_{kl}$ which is given by \eqref{c7tg} respectively.
\hfill\qed
\\

Apart from the above result of regularity on the Bloch spectrum, 
one has the first Bloch transform is an approximation to the Fourier transform.
\begin{proposition}[First Bloch transform to Fourier transform, \cite{CV}]\label{c7btof} 
Let $g^{\epsilon_1}$ and $g$ be in $L^2(\mathbb{R}^N)$. Then
\begin{enumerate}
\item[(i)] if $g^{\epsilon_1}\rightharpoonup g$ weakly in $L^2(\mathbb{R}_x^N)$, then 
$\chi_{\epsilon_1^{-1}Y_1^{\prime}}B_{1,A}^{\epsilon_1}\ g^{\epsilon_1} \rightarrow \widehat{g}$ weakly in $L^2_{loc}(\mathbb{R}^N_{\xi})$ provided
there is a fixed compact set $K$ such that $supp\ (g^{\epsilon_1}) \subset K \quad\forall\epsilon_1.$
\item[(ii)]  If $g^{\epsilon_1}\rightarrow g$ strongly in $L^2(\mathbb{R}_x^N)$, then 
$\chi_{\epsilon_1^{-1}Y_1^{\prime}}B_{1,A}^{\epsilon_1}\ g^{\epsilon_1} \rightarrow \widehat{g}$ strongly in $L^2_{loc}(\mathbb{R}^N_{\xi}).$ 
\end{enumerate}
\end{proposition}
\noindent
These results lead us to the following homogenization theorem in $\mathbb{R}^N$ established in \cite{CV}.
\begin{theorem}[Homogenization of the state equation, \cite{CV}]\label{c7ptf}
We consider the sequence $u^{\epsilon_1}$ satisfying $u^{\epsilon_1}\in H^1(\mathbb{R}^N)$ satisfying  
the equation $\mathcal{A}^{\epsilon_1}u^{\epsilon_1} = f $ in $\mathbb{R}^N$ where $f\in L^{2}(\mathbb{R}^N)$
with the fact that, $u^{\epsilon_1}\rightharpoonup u$ in $H^{1}(\mathbb{R}^N)$ weak and $u^{\epsilon_1}\rightarrow u$
in $L^{2}(\mathbb{R}^N)$ strong. We can express the equation in the equivalent form, then
\begin{equation*} 
\sigma^{\epsilon_1}_{k}=\ a^{\epsilon_1}_{kl}\frac{\partial u^{\epsilon_1}}{\partial x_l} \rightharpoonup a^{*}_{kl}\frac{\partial u}{\partial x_l}=\ \sigma_k \mbox{ in }L^2(\mathbb{R}^N)\quad\forall k =1,2,..,N.
\end{equation*}
In particular, $u$ satisfies $\mathcal{A}^{*}u =\ -\frac{\partial}{\partial x_l}(a^{*}_{kl}\frac{\partial}{\partial x_k}u)= f $ in $\mathbb{R}^N$. 
\hfill\qed
\end{theorem}
\noindent 
Once the homogenization result in $\mathbb{R}^N$ is established, it is easy  to
deduce the corresponding result in a bounded domain $\Omega$ by localization techniques
using a cut-off function $\varphi\in D(\Omega)$.\\

Using all these above tools defined so far, we will establish the limit system \eqref{c7Sd1}
in our final analysis.  
\section{Homogenization result}\label{c7en5}
\setcounter{equation}{0}
In this section, we derive of the main result of homogenization stated below.
It will be based on the tools whatever we have discussed in the previous section. 
\begin{theorem}\label{c7pst}
Let us consider $\Omega$ be an open set in $\mathbb{R}^N$. Let $Y_1$ and $Y_2$ are two periodic cell, then we define the operators
\begin{equation*}
\mathcal{A}^{\epsilon} = -\frac{\partial}{\partial x_k}(a^{\epsilon}_{kl}\frac{\partial}{\partial x_l})\quad\mbox{in }\Omega,\ \mbox{ where  }a^{\epsilon}_{kl}(x) = a(\frac{x}{\epsilon})_{kl} \mbox{ in } \epsilon Y_1 \mbox{ ; }\ \Omega\approx \cup\epsilon Y_1
\end{equation*}
and 
\begin{equation*}\begin{aligned}
&\mathcal{B}^{\epsilon} = -\frac{\partial}{\partial x_k}(b^{\epsilon}_{kl}\frac{\partial}{\partial x_l})\quad\mbox{in }\Omega,\ \mbox{ with }b^{\epsilon}_{kl}(x) = b(\frac{x}{h(\epsilon)})_{kl} \mbox{ in } h(\epsilon)Y_2; \\
&\mbox{ where }h(\epsilon)\rightarrow 0 \mbox{ as }\epsilon \rightarrow 0 \mbox{ considered as a new scale with }\Omega\approx \cup h(\epsilon)Y_2. 
\end{aligned}\end{equation*}
\\
For a given $f \in L^2(\Omega),$ let $u^{\epsilon} \in H^1_0(\Omega)$ be the
unique solution of the state equation
\begin{equation*} \mathcal{A}^{\epsilon}u^{\epsilon} = f \mbox{ in }\Omega\end{equation*}
and $p^{\epsilon}\in H^1_0(\Omega)$ be the unique solution for the adjoint state equation  
\begin{equation*}\mathcal{A}^{\epsilon}p^{\epsilon} = \mathcal{B}^{\epsilon}u^{\epsilon} \quad\mbox{in }\Omega.\end{equation*}
Then there exists a $u\in H^1_{0}(\Omega)$ and $p\in H^1_0(\Omega)$ such that the sequences
$u^{\epsilon}$ and $p^{\epsilon}$ converges to $u$ and $p$ respectively in $H^1_{0}(\Omega)$ 
weak with the following convergence of fluxes
\begin{equation*}\sigma^{\epsilon}_{k}=\ a^{\epsilon}_{kl}\frac{\partial u^{\epsilon}}{\partial x_l} \rightharpoonup a^{*}_{kl}\frac{\partial u}{\partial x_l}=\ \sigma_{k} \mbox{ in }L^2(\Omega)\quad\forall k =1,2,..,N.\end{equation*}
In particular, $u$ satisfies $\mathcal{A}^{*}u=\ -\frac{\partial}{\partial x_k}(a^{*}_{kl}\frac{\partial}{\partial x_l}u)= f $ in $\Omega$ and,
\begin{equation*}z^{\epsilon}_{k}=\ a^{\epsilon}_{kl}\frac{\partial p^{\epsilon}}{\partial x_l} - B^{\epsilon}_{kl}\frac{\partial u^{\epsilon}}{\partial x_l} \rightharpoonup a^{*}_{kl}\frac{\partial p}{\partial x_l}- b^{\#}_{kl}\frac{\partial u}{\partial x_l}= z_{k} \mbox{ in }L^2(\Omega). \end{equation*}
Moreover they satisfy, 
\begin{equation*} \mathcal{A}^{*}p=\ -\frac{\partial}{\partial x_k}(a^{*}_{kl}\frac{\partial}{\partial x_l})p=\ -\frac{\partial}{\partial x_k}(b^{\#}_{kl}\frac{\partial}{\partial x_l})u=\ \mathcal{B}^{\#}u\quad\mbox{ in }\Omega.\end{equation*}
\end{theorem}
\begin{proof}[Proof of the Theorem \ref{c7pst}]
The first part of the theorem is finding the limit equation for the state $u$ which simply follows from the
work of \cite{CV} in particular the Theorem \ref{c7ptf} stated in the previously. \\
So we move into the second part for finding the limit equation for the adjoint state $p$. We start with the cut-off function technique to localize the equation.
\paragraph{Step 1. Localization :}
Let $v \in D(\Omega)$ be arbitrary. Then the localization $vp^{\epsilon}$ satisfies
\begin{equation}\label{c7local} \mathcal{A}^{\epsilon}(vp^{\epsilon}) = v\mathcal{B}^{\epsilon}(u^{\epsilon}) + g_1^{\epsilon} + h_1^{\epsilon}\quad\mbox{in }\mathbb{R}^N \end{equation}
where,
\begin{equation*}
g_1^{\epsilon}=\ -2a^{\epsilon}_{kl}\frac{\partial p^{\epsilon}}{\partial x_l}\frac{\partial v}{\partial x_k} - a^{\epsilon}_{kl}\frac{\partial^2v}{\partial x_k\partial x_l}p^{\epsilon},\quad h_1^{\epsilon}  =\ -\frac{\partial a^{\epsilon}_{kl}}{\partial x_k}\frac{\partial v}{\partial x_l}p^{\epsilon}
\end{equation*}
$g_1^{\epsilon}$ and $h_1^{\epsilon}$ correspond to terms containing zeroth and first order derivatives on $a^{\epsilon}_{kl}$ respectively.
And
\begin{equation}\label{c7local2}
v\mathcal{B}^{\epsilon}(u^{\epsilon}) = \mathcal{B}^{\epsilon}(vu^{\epsilon}) + g_2^{\epsilon} + h_2^{\epsilon}\quad\mbox{in }\mathbb{R}^N \end{equation}
where,
\begin{equation*}
g_2^{\epsilon} =\ 2b^{\epsilon}_{kl}\frac{\partial u^{\epsilon}}{\partial x_l}\frac{\partial v}{\partial x_k} + b^{\epsilon}_{kl}\frac{\partial^2v}{\partial x_k\partial x_l}u^{\epsilon},\quad h_2^{\epsilon}  =\ \frac{\partial b^{\epsilon}_{kl}}{\partial x_k}\frac{\partial v}{\partial x_l}u^{\epsilon}
\end{equation*}
$g_2^{\epsilon}$ and $h_2^{\epsilon}$
correspond to terms containing zeroth and first order derivatives on $b^{\epsilon}_{kl}$ respectively.
\paragraph{Step 2. Limit of the LHS of \eqref{c7local} :}
We consider the first Bloch transform $B_1^{\epsilon}(\xi)$ \eqref{c7bt} of the equation \eqref{c7local} and
determine the limit in the Fourier space. Let us first consider the LHS of \eqref{c7local}, 
after taking the first Bloch transformation
we get $\lambda_1^{\epsilon}(\xi)B_1^{\epsilon}(vp^{\epsilon})$. 
Since $v$ has compact support thus $vp^{\epsilon} \rightarrow vp$ 
in $L^2(\mathbb{R}^N)$ strong, so by using the Proposition \ref{c7btof} we get 
\begin{equation*}
\chi_{\epsilon^{-1}Y_1^{\prime}}(\xi)\lambda_1^{\epsilon}(\xi)B_1^{\epsilon}(vu^{\epsilon}) \rightarrow \frac{1}{2}D^2_{kl}\lambda_1(0)\xi_k\xi_l\widehat{vp}(\xi)\quad\mbox{in }L^2_{loc}(\mathbb{R}^N)\mbox{ strong.}
\end{equation*}
\paragraph{Step(3). Limit of $B_1^{\epsilon}(g_1^{\epsilon} +g_2^{\epsilon})$ :}
Since $z^{\epsilon}=a^{\epsilon}_{kl}\frac{\partial p}{\partial x_l}- b^{\epsilon}_{kl}\frac{\partial u}{\partial x_l}$ 
is bounded in $L^2(\Omega)$, 
there exists a convergent subsequence with limit $z\in L^2(\Omega)$ and
we extent it by zero outside $\Omega$.
Thus,
\begin{equation*}
g_1^{\epsilon} +g_2^{\epsilon} \rightharpoonup g 
= -2z_k\frac{\partial v}{\partial x_k} - (M_{Y_1}({a_{kl}})p - M_{Y_2}(b_{kl})u)\frac{\partial^2v}{\partial x_k\partial x_l} \quad\mbox{in }L^2_{loc}(\mathbb{R}^N)\mbox{ weak}.
\end{equation*}
where $M_{Y_1}({a_{kl}})$ and $M_{Y_2}({b_{kl}})$ is the $L^{\infty}$ weak* limit of $a^{\epsilon}_{kl}$
and $b^{\epsilon}_{kl}$ satisfying $M_{Y_1}({a_{kl}}) =\frac{1}{|Y_1|}\int_{Y_1} a_{kl}(y) dy$ 
and $M_{Y_2}({b_{kl}}) =\frac{1}{|Y_2|}\int_{Y_2} b_{kl}(y) dy$ respectively.\\
\\
Thus, by applying Proposition \ref{c7btof} we have
\begin{equation*}
\chi_{\epsilon^{-1}Y_1^{\prime}}(\xi)B_1^{\epsilon}(g_1^{\epsilon}+g^{\epsilon}_2)(\xi) \rightharpoonup \widehat{g}(\xi)
\quad\mbox{in }L^2(\mathbb{R}^N)\mbox{ weak}. 
\end{equation*}
As we see, through integration by parts
\begin{equation*}\begin{aligned}
\widehat{g}(\xi) = \frac{1}{|Y|^{1/2}}\int_{\mathbb{R}^N} &[-2z_k\frac{\partial v}{\partial x_k} + (M_{Y}(a_{kl})\frac{\partial p}{\partial x_k}-M_{Y}(b_{kl})\frac{\partial u}{\partial x_k})\frac{\partial v}{\partial x_l} \\
                 &- (i\xi_k)(M_{Y}(a_{kl})\frac{\partial v}{\partial x_l}p - M_{Y}(b_{kl})\frac{\partial v}{\partial x_l}u)] e^{-ix.\xi} dx.
\end{aligned}\end{equation*}
\paragraph{Step 4. Limit of $B_1^{\epsilon}(h_1^{\epsilon}+h_2^{\epsilon})(\xi)$ :}
Here we see that  $h_1^{\epsilon}$, $h_2^{\epsilon}$ are
uniformly supported in a fixed compact set and
bounded in $H^{-1}(\mathbb{R}^N)$ but not in $L^2(\mathbb{R})$. 
So in order to calculate $B_1^{\epsilon}(h_1^{\epsilon}+h_2^{\epsilon})(\xi)$ we use the idea of 
decomposition.
\begin{equation}\begin{aligned}\label{c7abv0}
B_1^{\epsilon}(h_1^{\epsilon}+ h_2^{\epsilon})(\xi) = 
&\int_{\mathbb{R}^N}( h_1^{\epsilon}+ h_2^{\epsilon})(x)e^{-ix.\xi}\overline{\varphi_1}(\frac{x}{\epsilon};0)dx \\
&+ \int_{\mathbb{R}^N}( h_1^{\epsilon}+h_2^{\epsilon})(x)e^{-ix.\xi}\lb\overline{\varphi_1}(\frac{x}{\epsilon},\epsilon\xi)- \overline{\varphi_1}(\frac{x}{\epsilon};0)\rb dx . 
\end{aligned}\end{equation}
We start with the second term of the RHS, by following the Taylor expansion of $\varphi_1(y;\eta)$ we get
\begin{equation*}
-\int_{\mathbb{R}^N}(\frac{\partial a^{\epsilon}_{kl}}{\partial x_k}p^{\epsilon}-\frac{\partial b^{\epsilon}_{kl}}{\partial x_k}u^{\epsilon})\frac{\partial v}{\partial x_l}  
e^{-ix.\xi}\lb \epsilon\frac{\partial\overline{\varphi_1}}{\partial \eta_j}(\frac{x}{\epsilon};0)\xi_j + \mathcal{O}(\epsilon^2\xi^2)\rb dx 
\end{equation*}
which via integrating by parts becomes 
\begin{equation}\label{c7abv1}
\xi_j \int_{\mathbb{R}^N} (a^{\epsilon}_{kl}p^{\epsilon}-b^{\epsilon}_{kl}u^{\epsilon})\frac{\partial v}{\partial x_l}e^{-ix.\xi}
\frac{\partial^2\overline{\varphi_1}}{\partial \eta_j\partial y_k}(\frac{x}{\epsilon};0) dx  + \mathcal{O}(\epsilon\xi).
\end{equation}
The above integral term converges in $L^2_{loc}(\mathbb{R}^N)$ to
\begin{equation}\label{c7abv2}
\lb\frac{1}{|Y_1|}\int_{Y_1}a_{kl}(y)\frac{\partial^2\overline{\varphi_1}}{\partial \eta_j\partial y_k}(y;0)dy\rb \xi_j \int_{\mathbb{R}^N} \frac{\partial v}{\partial x_l}p e^{-ix.\xi} dx -L\xi_j \int_{\mathbb{R}^N} \frac{\partial v}{\partial x_l}u e^{-ix.\xi} dx
\end{equation}
where \begin{equation*}
L = L^{\infty}\mbox{ weak* limit }\lb b_{kl}(\frac{x}{f(\epsilon)})\frac{\partial^2\overline{\varphi_1}}{\partial \eta_j\partial y_k}(\frac{x}{\epsilon};0)\rb
\end{equation*}
If $\frac{\epsilon}{f(\epsilon)} \rightarrow t \in [0,\infty)$ as $\epsilon \rightarrow 0$, then  
\begin{equation*}
L =  \frac{1}{|Y_1|}\int_{Y_1}b_{kl}(ty)\frac{\partial^2\overline{\varphi_1}}{\partial \eta_j\partial y_k}(y;0)dy
\end{equation*}
and if $\frac{f(\epsilon)}{\epsilon} \rightarrow s \in [0,\infty)$ as $\epsilon \rightarrow 0$, then  
\begin{equation*}
L =  \frac{1}{|Y_2|}\int_{Y_2}b_{kl}(y)\frac{\partial^2\overline{\varphi_1}}{\partial \eta_j\partial y_k}(sy;0)dy.
\end{equation*}
Let $I_1^{\epsilon}(x) = a^{\epsilon}_{kl}\frac{\partial v}{\partial x_l}p^{\epsilon}e^{-ix.\xi}
\frac{\partial^2\overline{\varphi_1}}{\partial \eta_k\partial \eta_l}(\frac{x}{\epsilon};0)\in L^1(\mathbb{R}^N_x)$
and $I_2^{\epsilon}(x) = b^{\epsilon}_{kl}\frac{\partial v}{\partial x_l}u^{\epsilon}e^{-ix.\xi}
\frac{\partial^2\overline{\varphi_1}}{\partial \eta_k\partial \eta_l}(\frac{x}{\epsilon};0)
\in L^1(\mathbb{R}^N_x)$
then the above integrand term \eqref{c7abv1} is $\widehat{(I_1^{\epsilon}-I_2^{\epsilon})}(\xi)\in L^{\infty}_{loc}(\mathbb{R}^N_{\xi})$
and consequently $\widehat{(I_1^{\epsilon}-I_2^{\epsilon})}(\xi)\rightarrow \widehat{I}(\xi)$ $\forall \xi\in\mathbb{R}^N$
where $I(x)$  is given by \eqref{c7abv2}. Thus $\xi_j\widehat{(I_1^{\epsilon}-I_2^{\epsilon})}(\xi)\rightarrow \xi_j\widehat{I}(\xi)$ in $L^2_{loc}(\mathbb{R}^N)$ strongly.\\
Now consider the first term of the RHS of \eqref{c7abv0}, after doing integration by parts, 
one has
\begin{align*}
\frac{1}{|Y_1|^{\frac{1}{2}}}\int_{\mathbb{R}^N} a^{\epsilon}_{kl} 
[\frac{\partial^2v}{\partial x_k\partial x_l}p^{\epsilon} + \frac{\partial v}{\partial x_l}\frac{\partial p^{\epsilon}}{\partial x_k} &-i\xi_k\frac{\partial v}{\partial x_l}p^{\epsilon}]e^{-ix.\xi} dx  \\
-&\frac{1}{|Y_1|^{\frac{1}{2}}}\int_{\mathbb{R}^N} b^{\epsilon}_{kl} 
[\frac{\partial^2v}{\partial x_k\partial x_l}u^{\epsilon} + \frac{\partial v}{\partial x_l}\frac{\partial u^{\epsilon}}{\partial x_k} -i\xi_k\frac{\partial v}{\partial x_l}u^{\epsilon}]e^{-ix.\xi} dx.  
\end{align*}
By the similar way as we just have done the limit of the above equation would be
\begin{equation*}\begin{aligned}
\frac{1}{|Y_1|^{\frac{1}{2}}}\int_{\mathbb{R}^N}[ M_{Y_1}({a}_{kl}) 
\frac{\partial^2v}{\partial x_k\partial x_l}p &+ z_l\frac{\partial v}{\partial x_l} - (i\xi_k) M_{Y_1}(a_{kl})\frac{\partial v}{\partial x_l}p]e^{-ix.\xi} dx  \\
\frac{1}{|Y_1|^{\frac{1}{2}}}&\int_{\mathbb{R}^N}[ M_{Y_2}({b}_{kl}) 
\frac{\partial^2v}{\partial x_k\partial x_l}u -(i\xi_k) M_{Y_1}(b_{kl})\frac{\partial v}{\partial x_l}u]e^{-ix.\xi} dx  
\end{aligned}\end{equation*}
and again performing the integration by parts in the first term, we see the above is equal to
\begin{equation}\label{c7NE2}
\frac{1}{|Y_1|^{\frac{1}{2}}}\int_{\mathbb{R}^N}\lb z_l\frac{\partial v}{\partial x_l} 
- M_{Y_1}(a_{kl})\frac{\partial v}{\partial x_l}\frac{\partial p}{\partial x_k} 
+ M_{Y_2}(b_{kl})\frac{\partial v}{\partial x_l}\frac{\partial u}{\partial x_k}\rb e^{-ix.\xi}dx 
\end{equation}
Now combining \eqref{c7NE2} and \eqref{c7abv2} and using the fact 
$\frac{\partial\overline{\varphi_1}}{\partial \eta_j}(y;0) = -i|Y_1|^{-1/2}(\chi^{k}(y))$,
we see that $\chi_{\epsilon}^{-1}Y^{\prime}B_1^{\epsilon}(h_1^{\epsilon}+h_2^{\epsilon})(\xi)$ converges 
strongly in $L^2_{loc}(\mathbb{R}^N)$ to
\begin{equation*}\begin{aligned}
&-|Y_1|^{-1/2}M_{Y_1}\lb a_{kl}\frac{\partial(\chi^{j}(y))}{\partial y_k}\rb(i\xi_j)\int_{\mathbb{R}^N} \frac{\partial v}{\partial x_l}p e^{-ix.\xi}dx \\
&+|Y_1|^{-1/2}\int_{\mathbb{R}^N}\lb z_l \frac{\partial v}{\partial x_l} - M_{Y_1}(a_{kl})\frac{\partial v}{\partial x_l}\frac{\partial p}{\partial x_k}\rb e^{-ix.\xi} dx\\
&+|Y_1|^{-1/2} L\cdot(i\xi_j)\int_{\mathbb{R}^N} \frac{\partial v}{\partial x_l}u e^{-ix.\xi}dx +|Y_1|^{-1/2}\int_{\mathbb{R}^N}\lb M_{Y_2}(b_{kl})\frac{\partial v}{\partial x_l}\frac{\partial u}{\partial x_k}\rb e^{-ix.\xi} dx.
\end{aligned}\end{equation*}
\paragraph{Step 5. Limit of $\mathcal{B}^{\epsilon}(vu^{\epsilon})$ :}
\begin{equation}\label{c7bsint}\begin{aligned}
B_1^{\epsilon}\lb (-\frac{\partial}{\partial x_k}(b^{\epsilon}_{kl}\frac{\partial}{\partial x_l})vu^{\epsilon}\rb 
&=\ \int_{\mathbb{R}^N} (-\frac{\partial}{\partial x_k}b_{kl} ^{\epsilon}\frac{\partial}{\partial x_l}(vu^{\epsilon})) e^{-i x\xi}\overline{\varphi_1}(\frac{x}{\epsilon},\epsilon\xi) dx\\ 
&=\ \int_{\mathbb{R}^N}(vu^{\epsilon})(\frac{\partial}{\partial x_k}b_{kl} ^{\epsilon}\frac{\partial}{\partial x_l}(e^{-i x\xi}\overline{\varphi_1}(\frac{x}{\epsilon},\epsilon\xi))) dx. 
\end{aligned}\end{equation}
As we have
\begin{equation*}\frac{\partial}{\partial x_k}b_{kl}^{\epsilon}\frac{\partial}{\partial x_l}(e^{-ix.\xi}\overline{\varphi_1}(\frac{x}{\epsilon},\epsilon\xi))= \lb(\frac{\partial}{\partial x_k}+ i\xi_k)(b_{kl}^{\epsilon}(\frac{\partial}{\partial x_l}+i\xi_l)(\overline{\varphi_1}(\frac{x}{\epsilon},\epsilon\xi))\rb e^{-ix.\xi},\end{equation*}
and using the Bloch decomposition \eqref{c7ibt} we write $vu^{\epsilon}$ 
\begin{align*}
vu^{\epsilon}&=\ \int_{\epsilon^{-1}Y_1^{\prime}}\sum_{m=1}^{\infty} B^{\epsilon}_m (vu^{\epsilon})(\xi)e^{ix\cdot\xi} \varphi^{\epsilon}_m(x;\xi)d\xi\\
            &=\ \int_{\epsilon^{-1}Y_1^{\prime}}B^{\epsilon}_1(vu^{\epsilon})(\xi)e^{ix\cdot\xi} \varphi^{\epsilon}_1(x;\xi)d\xi + o(\epsilon)\quad\mbox{(by using Proposition }\ref{c7hm2} ).
\end{align*}
Then plugging these into the RHS of \eqref{c7bsint}, we get
\begin{align*}
B_1^{\epsilon}&\lb (-\frac{\partial}{\partial x_k}(b^{\epsilon}_{kl}\frac{\partial}{\partial x_l})vu^{\epsilon}\rb \\
 &=\ \int_{\mathbb{R}^N}\int_{\epsilon^{-1}Y_1^{\prime}}B^{\epsilon}_1(vu^{\epsilon})(\xi)b_{kl}^{\epsilon}(x)(\frac{\partial}{\partial x_k}+ i\xi_k)\varphi_1(\frac{x}{\epsilon};\epsilon\xi)(\frac{\partial}{\partial x_l}+i\xi_l)(\overline{\varphi_1}(\frac{x}{\epsilon},\epsilon\xi)) dx\ d\xi + o(\epsilon) 
\end{align*}
Now using the Taylor expansion of $\varphi_1(.;\epsilon\xi)$ at $0$ it is easy to see that the integral will converge to 
\begin{equation*}
\frac{1}{|Y_1|}\int_{Y_1}\lb b_{kl}(ty)(1+\frac{\partial}{\partial y_k}\chi_k(y))(1+\frac{\partial}{\partial y_k}\chi_l(y))\rb \xi_k\xi_l\widehat{vu}
\end{equation*}
when $a^{\epsilon}_{kl} =a_{kl}(\frac{x}{\epsilon})$ and $b_{kl}^{\epsilon} = b_{kl}(\frac{x}{f(\epsilon)})$ with the fact  $\frac{\epsilon}{f(\epsilon)} \rightarrow t \in [0,\infty)$ as $\epsilon \rightarrow 0$,\\
\\
or it converges to
\begin{equation*}
\frac{1}{|Y_2|}\int_{Y_2}\lb b_{kl}(y)(1+\frac{\partial}{\partial y_k}\chi_k(sy))(1+\frac{\partial}{\partial y_k}\chi_l(sy))\rb \xi_k\xi_l\widehat{vu} 
\end{equation*}
when $\frac{f(\epsilon)}{\epsilon} \rightarrow s \in [0,\infty)$ as $\epsilon \rightarrow 0$.
\paragraph{Step 5. Limit of \eqref{c7local} :}
Let $\mathcal{A}^{*} \equiv -\frac{\partial}{\partial x_k}(a^{*}_{kl}\frac{\partial}{\partial x_l})$
and  $\mathcal{B}^{\#} \equiv -\frac{\partial}{\partial x_k}(b^{\#}_{kl}\frac{\partial}{\partial x_l})$ then
\begin{equation}\begin{aligned}\label{c7lch}
\widehat{{\mathcal{A}^{*}(vp)}}(\xi) &- \widehat{\mathcal{B}^{\#}(vu)}(\xi)= - |Y_1|^{-1/2}\int_{\mathbb{R}^N} z_k\frac{\partial v}{\partial x_k}e^{-ix.\xi} dx \\
-(i\xi_k)|Y_1|^{-1/2}a^{*}_{kl}&\int_{\mathbb{R}^N} p\frac{\partial v}{\partial x_l} e^{-ix.\xi} dx +(i\xi_k)|Y_1|^{-1/2}b^{\#}_{kl}\int_{\mathbb{R}^N} u\frac{\partial v}{\partial x_l} e^{-ix.\xi} dx
\end{aligned}\end{equation}
The above equation is to be considered as the localized homogenized equation in the Fourier space. 
The conclusion of the Theorem will follow as a consequence of this equation.
\paragraph{Step 6. Fourier space ($\xi)$ to physical space $(x)$}:\\
\\
We take the inverse Fourier
transform of the localized homogenized equation \eqref{c7lch} to go back to physical space
\begin{equation*}\mathcal{A}^{*}(vp)- \mathcal{B}^{\#}(vu)= - z_k\frac{\partial v}{\partial x_k}
-a^{*}_{kl}\frac{\partial}{\partial x_k}(\frac{\partial v}{\partial x_l}p)  
+b^{\#}_{kl}\frac{\partial}{\partial x_k}(\frac{\partial v}{\partial x_l}u)
\quad\mbox{in }\mathbb{R}^N \end{equation*}
or,
\begin{equation*}v\lb \frac{\partial}{\partial x_k } z - \frac{\partial}{\partial x_k} (A^{*}\nabla p -B^{\#}\nabla u ) \rb = \lb z-  (A^{*}\nabla p -B^{\#}\nabla u ) \rb \frac{\partial v}{\partial x_k} \end{equation*}
As this relation holds for all $v\in D(\Omega)$, so the desired conclusion follows. 
Indeed, let us choose $v(x) = v_0(x)e^{inx·\omega}$ where $\omega$ is a unit vector in $\mathbb{R}^N$ and $v_0 \in D(\Omega)$ is fixed.
Letting $n\rightarrow \infty$ in the resulting relation and varying the unit vector $\omega$, we can easily deduce 
successively that $z_k = a^{*}_{kl}\frac{\partial p}{\partial x_l} - b^{\#}_{kl}\frac{\partial u}{\partial x_l}$ and $\mathcal{A}^{*}p = \mathcal{B}^{\#}u$ in $\Omega$.
This completes our discussion of the theorem of Homogenization result.
\hfill\end{proof}

\paragraph{Acknowledgement :}
This work has been carried out within a project supported by the 
Airbus Group Corporate Foundation Chair ``Mathematics of Complex Systems'' 
established at Tata Institute Of fundamental Research (TIFR) - Centre for Applicable Mathematics.

\bibliographystyle{plain}
\bibliography{Master_bibfile}
\end{document}